\documentclass[11pt]{article}
\usepackage{amssymb}
\usepackage{amsfonts}
\usepackage{amsmath}
\usepackage{color}
\usepackage[left=2cm,top=2cm,right=3cm]{geometry}
\newcommand{\beq}{\begin{equation}}
\newcommand{\eeq}{\end{equation}}
\newcommand{\bea}{\begin{eqnarray}}
\newcommand{\eea}{\end{eqnarray}}
\newcommand{\beas}{\begin{eqnarray*}}
\newcommand{\para}{\mathbin{\!/\mkern-5mu/\!}}
\newcommand{\eeas}{\end{eqnarray*}}

\newtheorem{theorem}{Theorem}[section]

\newtheorem{definition}[theorem]{Definition}
\newtheorem{proposition}[theorem]{Proposition}

\newtheorem{corollary}[theorem]{Corollary}
\newtheorem{lemma}[theorem]{Lemma}
\newtheorem{remark}[theorem]{Remark}
\newtheorem{example}[theorem]{Example}
\newtheorem{examples}[theorem]{Examples}
\newtheorem{foo}[theorem]{Remarks}

\newenvironment{proof}{\addvspace{\medskipamount}\par\noindent{\it
Proof}.}
{\unskip\nobreak\hfill$\Box$\par\addvspace{\medskipamount}}

\newcommand{\bM}{\mathbb M}

\newcommand{\M}{\mathbb M}

\newcommand{\ve}{\varepsilon}

\begin{document}
\title{Log-Sobolev inequalities on the horizontal path space of a totally geodesic foliation}
\author{Fabrice Baudoin, Qi Feng}
\date{Department of Mathematics, Purdue University \\
West Lafayette, IN, USA}
\maketitle

\begin{abstract}
We develop a Malliavin calculus on the horizontal path space of a totally geodesic Riemannian foliation. As a first application, under suitable assumptions, we prove a log-Sobolev inequality for a natural one-parameter  family of infinite-dimensional Ornstein-Uhlenbeck type operators. As a second application, we obtain concentration and tail estimates for the horizontal Brownian motion of the foliation.
\end{abstract}

\baselineskip 0.25in

\tableofcontents


\section{Introduction}

Log-Sobolev inequalities have extensively been studied in connection with convergence to equilibrium for parabolic partial differential equations and hypercontractivity properties of the corresponding semigroup. In \cite{BB}, log-Sobolev inequalities for subelliptic diffusion operators were studied by using purely analytic tools and then recovered in \cite{FBb} by using stochastic analysis. In the present work, we show that those log-Sobolev inequalities actually are  finite-dimensional projections of a family of  log-Sobolev inequalities that hold on the path space lying above the diffusion process associated to the subelliptic diffusion operator. To prove these inequalities, we develop a Malliavin calculus on the horizontal path space of a foliation. This calculus is interesting in itself, and besides log-Sobolev inequalities, it will also allow us to prove concentration properties and tail estimates. We will also single out an interesting one-parameter family of interesting Ornstein-Uhlenbeck type operators on the horizontal path space. Our log-Sobolev inequalities are then equivalent to hypercontractivity properties of the heat  semigroups generated by those operators.

\

We now enter into more details of our contribution. In Section 2, we present the prerequisites for the reading of the paper. Most of the material here is taken from the course \cite{FBa} to which we refer for further details. Our framework consists of a totally geodesic Riemannian foliation and we are interested in the horizontal Laplacian of the foliation. From our assumptions, this horizontal Laplacian is a hypoelliptic operator that (locally) satisfies the H\"ormander's bracket generating condition. As was proved in \cite{FBd}, one can prove Weitzenb\"ock type identities for this horizontal Laplacian. A fundamental difference with the Riemannian case, where Weitzenb\"ock identity involves one and only one operator on one-forums, the Hodge-de Rham operator, is that in this foliated case, there is a one-parameter family of canonical operators on one-forms. These Weitzenb\"ock identities admit several corollaries. They imply, in particular an integration by parts formula, for the horizontal Brownian motion, that is the diffusion generated by the horizontal Laplacian.

\

In Section 3, we enter into the heart of our analysis. We first introduce the relevant Malliavin derivatives on the horizontal path space and prove that the integration by parts formula for the horizontal Brownian motion is actually the projection of an integration by parts formula that holds on the path space of the diffusion. This integration by parts formula goes hand in hand with a Clark-Ocone type representation for functionals of the horizontal Brownian motion.

\

In Section 4, we adapt a method of E. P. Hsu \cite{Hsu}, to prove in our framework and from our Clark-Ocone representation, a log-Sobolev inequality on the path space of the horizontal Brownian motion.

\

In Section 5, as an application of our log-Sobolev inequalities, by adapting a method of M. Ledoux,  we provide concentration and tail estimates for the horizontal Brownian motion.

\section{Preliminaries}

Let $\mathbb{M}$  be a smooth and connected $n+m$ dimensional manifold. We assume that $\mathbb{M}$ is endowed with a Riemannian foliation with a complete bundle-like metric $g$ and totally geodesic $m-$dimensional leaves. We refer to the book  of Tondeur \cite{PT} for a detailed account about foliations. Since our framework is similar to the one in \cite{FBa, FBd}, we also refer to these references for further details about the setting and notations.

\

The sub-bundle $\mathcal{V}$ defined  from vectors tangent to the leaves is the set of $vertical ~directions$. The sub-bundle $\mathcal{H}$ which is normal to $\mathcal{V}$ is  the set of $horizontal~directions$. We will assume that  $\mathcal{H}$ is bracket generating. 

 The metric $g$ can be split as 
\[
g=g_{\mathcal{H}}\oplus g_{\mathcal{V}}.
\]
and the canonical variation of $g$ is defined as the one-parameter family of Riemannian metrics:
\[
g_{\varepsilon}=g_{\mathcal{H}}\otimes\frac{1}{\varepsilon}g_{\mathcal{V}},~~\varepsilon>0.
\]

The sub-Riemannian limit is $\varepsilon \to 0$. The Bott connection on $\M$ is defined as follows:
\[
\nabla_XY=\begin{cases} \pi_{\mathcal{H}}(\nabla_X^RY),X,Y\in \Gamma^{\infty}(\mathcal{H})\\ \pi_{\mathcal{H}}([X,Y]),X\in \Gamma^{\infty}(\mathcal{V}),Y\in \Gamma^{\infty}(\mathcal{H})\\\pi_{\mathcal{V}}([X,Y]),X\in \Gamma^{\infty}(\mathcal{H}),Y\in \Gamma^{\infty}(\mathcal{V})\\ \pi_{\mathcal{V}}(\nabla_X^RY),X,Y\in \Gamma^{\infty}(\mathcal{V})\end{cases}
\]
where $\nabla^R$ is the Levi-Civita connection of $g$ and $\pi_{\mathcal{H}}$ (resp. $\pi_{\mathcal{V}}$) the projection on $\mathcal{H}$ (resp. $\mathcal{V}$). It is easy to check that for every $\varepsilon>0$, this connection satisfies $\nabla g_{\varepsilon}=0.$

For $Z\in \Gamma^{\infty}(T\mathbb{M})$, there is a unique skew-symmetric endomorphism $J_Z:\mathcal{H}_x\rightarrow \mathcal{H}_x$ such that for all horizontal vector fields $X$ and $Y$, 
\begin{equation}
g_{\mathcal{H}}(J_Z(X),Y)=g_{\mathcal{V}}(Z,T(X,Y)).
\end{equation}
where $T$ is the torsion tensor of $\nabla$. We then extend $J_Z$ to be $0$ on $\mathcal{V}_x$. If $Z_1,\cdots,Z_m$ is a local vertical frame, the operator $\sum_{l=1}^mJ_{Z_{l}}J_{Z_{l}}$ does not depend on the choice of the frame and shall concisely be denoted by $\mathbf{J}^2$. 
The horizontal divergence of the torsion $T$ is the $(1,1)$ tensor which is defined in a local horizontal frame $X_1,\cdots,X_n$ by
\[
\delta_{\mathcal{H}}T(X)=\sum_{j=1}^n(\nabla_{X_j}T)(X_j,X).
\]

\begin{definition}
The foliation is said to be of Yang-Mills type if $\delta_{\mathcal{H}}T=0$.
\end{definition}

In this paper we always assume that the foliation is of Yang-Mills type.

We define  the horizontal Ricci curvature $\mathfrak{Ric}_{\mathcal{H}}$ as the fiberwise linear map from the space of one-form into itself which is such that for every $f,g \in C^\infty(\M)$,
\[
\langle \mathfrak{Ric}_{\mathcal{H}}(df),dg\rangle=\mathbf{Ricci}(\nabla_{\mathcal{H}}f,\nabla_{\mathcal{H}}g),
\]
where $\mathbf{Ricci}$ is the Ricci curvature of the Bott connection $\nabla$.

If $V$ is a horizontal vector field and $\varepsilon>0$, we consider the fiberwise linear map from the space of one-forms into itself which is given by $\eta\in \Gamma^{\infty}(T^*\mathbb{M})$ and $Y\in \Gamma^{\infty}(T\mathbb{M})$ by
\[
\mathfrak{T}_V^{\varepsilon}\eta(Y)=\begin{cases}\frac{1}{\varepsilon}(J_YV),~Y\in \Gamma^{\infty}(\mathcal{V})\\-\eta(T(V,Y)),~Y\in \Gamma^{\infty}(\mathcal{H})
\end{cases}\]
 We observe that $\mathfrak{T}_V^{\varepsilon}$ is skew-symmetric for the metric $g_{\varepsilon}$ so that $\nabla-\mathfrak{T}^{\varepsilon}$ is a $g_{\varepsilon}$-metric connection.

 Finally, we denote by $L$ the  horizontal Laplacian acting on functions or one-forms. In a local horizontal frame $X_i$, we have 
\[
L=\sum_{i=1}^n\nabla_{X_i}\nabla_{X_i}-\nabla_{\nabla_{X_i}X_i}.
\]

We now introduce a family of operators that was first considered in \cite{FBd}. For $\varepsilon>0$, we consider the following operator which is defined on one-forms by
\[
\square_{\varepsilon}=-(\nabla_{\mathcal{H}}-\mathfrak{T}_{\mathcal{H}}^{\varepsilon})^*(\nabla_{\mathcal{H}}-\mathfrak{T}^{\varepsilon}_{\mathcal{H}})-\frac{1}{\varepsilon}\mathbf{J}^2-\mathfrak{Ric}_{\mathcal{H}}.
\]
The following  result was proved in   \cite{FBd}.

\begin{proposition}[Theorem 3.1 in \cite{FBd}]\label{commu1}
For every $f\in C^{\infty}(\mathbb{M})$ and $\varepsilon >0$, we have
\[
dLf=\square_{\varepsilon}df
\]
\end{proposition}

The completeness of the metric $g$ implies that the horizontal Laplacian $L$ is essentially self-adjoint on the space of smooth and compactly supported functions. As such it generates a 
sub-Markov semigroup $P_t=e^{\frac{1}{2} tL}$.

\

 From now on and throughout the paper, we will assume that for every horizontal one-form $\eta\in \Gamma^{\infty}(\mathcal{H}^*)$,
\begin{equation}\label{bounds}
\langle \mathfrak{Ric}_{\mathcal{H}}\eta,\eta\rangle_{\mathcal{H}}\geq -K\|\eta\|^2_{\mathcal{H}}, ~~~-\langle\mathbf{J}^2(\eta),\eta\rangle_{\mathcal{H}}\leq \kappa \|\eta\|_{\mathcal{H}}^2,
\end{equation}
with $K \ge 0,\kappa >0$.

By using the argument in Lemma 4.3 of \cite{FBb}) we see that the operator $\square_{\varepsilon}$ is essentially self-adjoint on the space of smooth and compactly supported one-forms.  We denote $Q_t^\varepsilon$ the semigroup generated by $\frac{1}{2}\square_{\varepsilon}$ .  

The following result was then proved in \cite{FBd}.

\begin{proposition}[Lemma 4.1 in \cite{FBd}]
Let $\varepsilon >0$. If $f\in C_0^{\infty}(\mathbb{M})$, then for every $t\geq 0$,
\[
dP_tf=Q_t^{\varepsilon}df.
\]
\end{proposition}

We now turn to the stochastic representation of the semigroups we are interested in.  We first observe that from Theorem 4.2 in \cite{FBd}, the semigroup $P_t$ is stochastically complete. We denote by $(X_t)_{t\geq 0}$  the symmetric diffusion process generated by $\frac{1}{2}L$. The lifetime of the process is $\infty$.

Consider the process $\tau_t^{\varepsilon}:T_{X_t}^*\mathbb{M}\rightarrow T^*_{X_0}\mathbb{M}$ which is the solution of the following covariant Stratonovitch stochastic differential equation:
\begin{equation}\label{tau_t}
d[\tau_t^{\varepsilon}\alpha(X_t)]=\tau_t^{\varepsilon}\left( \nabla_{\circ dX_t}-\mathfrak{T}_{\circ dX_t}^{\varepsilon}-\frac{1}{2}
\left(\frac{1}{\varepsilon}\mathbf{J}^2 +\mathfrak{Ric}_{\mathcal{H}}\right)dt\right) \alpha(X_t),~~\tau_0^{\varepsilon}=\mathbf{Id},
\end{equation}
where $\alpha$ is any smooth one-form. It is easily seen that we have 
\begin{equation}\label{tau=M Theta}
\tau^{\varepsilon}_t=\mathcal{M}_{t}^{\varepsilon}\Theta_{t}^{\varepsilon}
\end{equation}
where the process $\Theta_t^{\varepsilon}: T_{X_t}^{*}\mathbb{M}\rightarrow T_{X_0}^{*}\mathbb{M}$ is the solution of the following covariant Stratonovitch stochastic differential equation: 
\begin{equation}\label{Theta equation}
d[\Theta_t^{\varepsilon}\alpha(X_t)]= \Theta_t^{\varepsilon}(\nabla_{\circ dX_t}-\mathfrak{L}^{\varepsilon}_{\circ dX_t})\alpha(X_t),~~\Theta_0^{\varepsilon}=\mathbf{Id}
\end{equation}
 where $\alpha$ is any smooth one-form. The multiplicative functional $(\mathcal{M}_t^{\varepsilon})_{t\geq 0}$ is the solution of the following equation
\begin{equation}\label{multiplicative function M_t}
\frac{d\mathcal{M}_t^{\varepsilon}}{dt}=-\frac{1}{2}\mathcal{M}_t^{\varepsilon}\Theta_t^{\varepsilon}\left(\frac{1}{\varepsilon}\mathbf{J}^2+\mathfrak{Ric}_{\mathcal{H}} \right)(\Theta_t^{\varepsilon})^{-1}, ~~\mathcal{M}_0^{\varepsilon}=\mathbf{Id}.
\end{equation} 

Since $\mathfrak{T}^{\varepsilon}$ is skew-symmetric,  $\Theta_t^{\varepsilon}$ is an isometry for the Riemannian metric $g_{\varepsilon}$. We deduce then from the assumption \eqref{bounds} that the following pointwise bound holds
\[
\| \tau^{\varepsilon}_t \|_\varepsilon \le e^{ \frac{1}{2} \left(K+\frac{\kappa}{\varepsilon} \right) t}.
\]

Combining the above Proposition 2.3 and applying Theorem 4.4 in \cite{FBb}, we get the the same representation for the derivative of the semigroup as Corolllary 4.5 in \cite{FBb}.

\begin{proposition}\label{dP_tf(X)}
Let $\varepsilon >0$. Let $f\in C_0^{\infty}(\mathbb{M})$. Then for every $t\geq0$, and $x\in \mathbb{M}$,
\[
dP_tf(x)=\mathbb{E}_x(\tau_t^{\varepsilon}df(X_t)).
\]
\end{proposition}

By following ideas in \cite{FBb}, this representation of the derivative of the semigroup allows to prove an integration by parts formula and a related Clark-Ocone representation. The details of the proof are let to the reader.

From now on, we fix a point $x \in \M$ as the starting point of the diffusion process $(X_t)_{t \ge 0}$. The stochastic parallel transport for the Bott connection $\nabla$ along the paths of $(X_t)_{t\geq 0}$ will be denoted by $\para_{0,t}.$ Since the connection $\nabla$ is horizontal, the map $\para_{0,t}:T_{x}\mathbb{M}\rightarrow T_{X_t}\mathbb{M}$ is an isometry that preserves the horizontal bundle, that is, if $u\in \mathcal{H}_{x}$, then $\para_{0,t}u\in \mathcal{H}_{X_t}$. The anti-development of $(X_t)_{t\geq 0}$, 
\[
B_t=\int_0^t\para^{-1}_{0,s}\circ dX_s,
\]
is a Brownian motion in the horizontal space $\mathcal{H}_{x}.$

\begin{proposition}\label{IBP at a point}
Let $\varepsilon >0$. For any $C^1$ and adapted process $\gamma :\mathbb{R}_{\geq 0}\rightarrow \mathcal{H}_x$ such that $\mathbb{E}_x(\int_0^{\infty}\|\gamma'(s)\|_{\mathcal{H}}^2ds)<\infty$ and any $f\in C_0^{\infty}(\mathbb{M}),~t\geq 0$,
\[
\mathbb{E}_x\left( f(X_t)\int_0^t\langle \gamma'(s),dB_s\rangle_{\mathcal{H}}\right)=\mathbb{E}_x\left(\left\langle\tau_t^{\varepsilon}df(X_t),\int_0^t(\tau_s^{\varepsilon,*})^{-1}\para_{0,s}\gamma'(s)ds\right\rangle \right).
\]
\end{proposition}

\begin{proposition}\label{clark-ocone n=1}
Let $\varepsilon >0$. For every $f\in C_0^{\infty}(\mathbb{M})$, and every $t>0,$
\[
f(X_t)=P_tf(x)+\int_0^t\left \langle \mathbb{E}_x((\tau_s^{\varepsilon})^{-1}\tau^{\varepsilon}_tf(X_t)|\mathcal{F}_s),\para_{0,s}dB_s \right \rangle
\]
\end{proposition}

\section{Malliavin calculus on the horizontal path space}

In the following we are going to generalized the Clark-Ocone formula in our new setting. Let us first recall what is a cylinder function.
\begin{definition}\label{cylinder}
A random variable $F$ is called a smooth cylinder function on the horizontal  path space $W_\mathcal{H} (\mathbb{M})$ of $\mathbb{M}$, if there exists a partition $\{0=t_0<t_1<t_2<,\cdots,<t_n\leq  T\}$ of $[0,T]$ and $f\in C_0^{\infty}(\mathbb{M}^{n})$ such that 
\begin{equation}\label{cylinder function}
F=f(X_{t_1},\cdots,X_{t_n}).
\end{equation}
The set of smooth cylinder functions is denoted by $\mathcal{C}$.
\end{definition}

Then,  inspired by an idea of Fang-Malliavin in \cite{Fang} and their definition of different gradients on the path space, we set the following definitions

\begin{definition}\label{gradient definition}
Let $\varepsilon >0$. For  $F=f(X_{t_1},\cdots,X_{t_n})$, $f\in C_0^{\infty}(\mathbb{M}^n)$, we define the
\begin{itemize}
\item \textbf{Intrinsic gradient:}
\[
D^\varepsilon_tF=\sum_{i=1}^n\mathbf{1}_{0,t_i}(t)\Theta_{t_i}^{\varepsilon}d_if(X_{t_1},\cdots,X_{t_n}), \quad 0 \le t \le T
\] 

\item \textbf{Damped gradient:}
\[
\tilde{D}_t^{\varepsilon}F=\sum_{i=1}^n\mathbf{1}_{[0,t_i]}(t)(\tau_t^{\varepsilon})^{-1}\tau_{t_i}^{\varepsilon}d_if(X_{t_1},\cdots,X_{t_n}), \quad  0 \le t \le T.
\]

\end{itemize}
\end{definition}

By using the intrinsic gradient, we can then define a one-parameter family of Ornstein-Uhlenbeck type operators by
\[
\mathcal{L}_\varepsilon=-(D^\varepsilon)^* D^\varepsilon
\]

Our main objective will be to prove an integration by parts formula for the damped gradient. We start with a preliminary lemma whose Riemmanian ancestor was  first proved by Hsu in \cite{Hsu}.

\begin{proposition}\label{gradient of expectation formula}
Let $\varepsilon >0$. Let  $F=f(X_{t_1},\cdots,X_{t_n})$, $f\in C^{\infty}_0(\mathbb{M}^n)$. We have
\[
d\mathbb{E}_x(F)=\mathbb{E}_x\left(\sum_{i=1}^n\tau_{t_i}^{\varepsilon}d_if(X_{t_1},\cdots,X_{t_n})\right),
\]
where the derivative $d$ is computed with respect to the starting point $x$.
\end{proposition}
\begin{proof}
We will proceed our proof by induction. Consider a cylinder function $F=f(X_{t_1},\cdots,X_{t_n})$. For $n=1$, the statement follows from  Proposition \ref{dP_tf(X)}. Now we assume that the proposition is true for the case $n-1$ and we are going to show the case $n$.
By the Markov Property, we have
\[
\mathbb{E}_x(F)=\mathbb{E}_x(g(X_{t_1})),~\text{where}~g(y)=\mathbb{E}_y(f(y,X_{t_2-t_1},\cdots,X_{t_n-t_1}))
\]
From  Proposition \ref{dP_tf(X)},  we have
\[
d\mathbb{E}_x(F)=\mathbb{E}(\tau_{t_1}^{\varepsilon}dg(X_{t_1})).
\]
By using the induction assumption, we have
\[
dg(y)=\mathbb{E}_y(d_1f(y,X_{t_2-t_1}))+\mathbb{E}_y(\sum_{i=2}^n\tau^{\varepsilon}_{t_i-t_1}d_if(y,X_{t_2-t_1}))
\]
Applying then the multiplicative property for $\tau_t^{\varepsilon}$ and the Markov property again, we obtain
\[
d\mathbb{E}_x(F)=\mathbb{E}_x\left(\sum_{i=1}^n\tau_{t_i}^{\varepsilon}d_if(X_{t_1},\cdots,X_{t_n})\right).
\]
\end{proof}

With this lemma in hands, we can now turn to the proof of one of our main results.

\begin{proposition}[Clark-Ocone formula] \label{clark-ocone}
Let $\varepsilon >0$. Let $F=f(X_{t_1},\cdots,X_{t_n})$, $f\in C^{\infty}_0(\mathbb{M}^n)$. Then
\[
F=\mathbb{E}_x(F)+\int_0^T \langle \mathbb{E}_x(\tilde{D}_s^{\varepsilon}F|\mathcal{F}_s),\para_{0,s}dB_s\rangle.
\]
\end{proposition}
\begin{proof}
We will still proceed the proof by induction. When $n=1$, it is true from Proposition \ref{clark-ocone n=1}. Now let us proceed to the case $n$ by assuming that the formula is true for the case $n-1$.
We first represent the function $F=f(X_{t_1},\cdots,X_{t_n})$ conditioned on starting from $X_{t_1}$ and use the case $n=1$.  We have
\begin{equation}\label{equation 3.7}
F
=\mathbb{E}_{X_{t_1}}(f(X_{t_1},\cdots,X_{t_n}))+\int_{t_1}^{t_{n}}\langle \mathbb{E}_x(\sum_{i=2}^n\mathbf{1}_{[t_1,t_i]}(t)(\tau_s^{\varepsilon})^{-1}\tau_{t_i-t_1}^{\varepsilon}d_if(X_{t_1},\cdots,X_{t_n}) )|\mathcal{F}_s),\para_{0,s}dB_s\rangle.
\end{equation}
By using the Markov property, we can then write
\[
\mathbb{E}_{X_{t_1}}(f(X_{t_1},\cdots,X_{t_n}))=g(X_{t_1}),~\text{where}~g(y)=\mathbb{E}_{y}(f(y,\cdots,X_{t_n-t_{1}}))
\] 
By using Proposition \ref{gradient of expectation formula}, we have
\[
g(X_{t_1})=\mathbb{E}_{x}(f(X_{t_1},\cdots,X_{t_{n-1}},X_{t_n}))+\int_0^{t_1}\langle\mathbb{E}_x((\tau_s^{\varepsilon})^{-1}\tau_{t_1}^{\varepsilon}dg(X_{t_1})|\mathcal{F}_s),\para_{0,s}dB_s\rangle.
\]
\[
=\mathbb{E}_{x}(f(X_{t_1}\cdots,X_{t_{n-1}},X_{t_n}))
+\int_0^{t_1}\langle\mathbb{E}_x((\tau_s^{\varepsilon})^{-1}\tau_{t_1}^{\varepsilon}d_1f(X_{t_1}\cdots,X_{t_n})|\mathcal{F}_s),\para_{0,s}dB_s\rangle.
\]
\[
+\int_0^{t_1}\langle\mathbb{E}_x((\tau_s^{\varepsilon})^{-1}\tau_{t_1}^{\varepsilon}\sum_{i=2}^n\mathbb{E}_{X_{t_1}}(\tau_{t_i-t_1}^{\varepsilon}d_if(X_{t_1}\cdots,X_{t_n}))|\mathcal{F}_s),\para_{0,s}dB_s\rangle.
\]
Thus by applying the multiplicative property of $\tau_t^{\varepsilon}$ and the Markov property, we obtain 
\[\mathbb{E}_{X_{t_1}}(f(X_{t_1},\cdots,X_{t_n}))
=\mathbb{E}_{x}(f(X_{t_1}\cdots,X_{t_{n-1}},X_{t_n}))
+\int_0^{t_1}\langle\mathbb{E}_x((\tau_s^{\varepsilon})^{-1}\tau_{t_1}^{\varepsilon}d_1f(X_{t_1}\cdots,X_{t_n})|\mathcal{F}_s),\para_{0,s}dB_s\rangle.
\]
\begin{equation}\label{3.8}
+\int_0^{t_1}\langle\mathbb{E}_x(\sum_{i=2}^n(\tau_s^{\varepsilon})^{-1}\tau_{t_i}^{\varepsilon}d_if(X_{t_1}\cdots,X_{t_n})|\mathcal{F}_s),\para_{0,s}dB_s\rangle.
\end{equation}
The proof is then completed by combining \eqref{equation 3.7} and \eqref{3.8}.
\end{proof}

An immediate corollary of the previous Clark-Ocone type representation is the following integration by parts formula on the horizontal path space.

\begin{corollary}
Let $\varepsilon >0$.  Let $F=f(X_{t_1},\cdots,X_{t_n})$, $f\in C^{\infty}_0(\mathbb{M}^n)$.  For any $C^1$ and adapted process $\gamma :[0,T] \rightarrow \mathcal{H}_x$ such that $\mathbb{E}_x(\int_0^{T}\|\gamma'(s)\|_{\mathcal{H}}^2ds)<\infty$,
\[
\mathbb{E}_x\left( F \int_0^T\langle \gamma'(s),dB_s\rangle_{\mathcal{H}}\right)=\mathbb{E}_x\left( \int_0^T\langle \tilde{D}_s^{\varepsilon}F ,\para_{0,s} \gamma'(s)\rangle ds \right).
\]

\end{corollary}

Let us remark that this corollary proves that $\tilde{D}^{\varepsilon}$ is closable.

\section{Log-Sobolev inequalities on the horizontal path space}

In this section, we prove a family of log-Sobolev inequalities on the horizontal path space of $\M$. To this end, we adapt to our setting a method which is due to Hsu \cite{Hsu} (see also \cite{Cap}).

Throughout the section we will assume that for every horizontal one-form $\eta$,
\begin{equation}\label{geometric assumption}
| \langle \mathfrak{Ric}_{\mathcal{H}} \eta, \eta  \rangle_\mathcal{H} | \leq K \| \eta \|^2_\mathcal{H} ,~~\langle \mathbf{J}^{*}\mathbf{J}\eta,\eta\rangle_{\mathcal{H}}\leq \kappa\|\eta\|_{\mathcal{H}}^2.
\end{equation}

\begin{theorem}\label{log-Sobolev inequality}
For every cylindric function $G\in \mathcal{C}$ we have the following log-Sobolev inequality.
\begin{equation*}
\mathbb{E}_x(G^2\ln G^2)-\mathbb{E}_x(G^2)\ln \mathbb{E}_x(G^2)\leq 2 e^{3T(K+\frac{\kappa}{\varepsilon})} \mathbb{E}_x \left( \int_0^T\|D^\varepsilon_sG\|_{\varepsilon}^2ds\right).
\end{equation*}
\end{theorem}

We split the proof in several lemmas.

\begin{lemma}\label{log sob damped}
We have the following inequality
\begin{equation*}
\mathbb{E}_x(G^2\ln G^2)-\mathbb{E}_x(G^2)\ln \mathbb{E}_x(G^2) \leq 2 \mathbb{E}_x \left(\int_0^T \|\tilde{D}^{\varepsilon}_sG\|_{\varepsilon}^2ds\right)
\end{equation*}
\end{lemma}
\begin{proof}
Let us consider the martingale $N_s= \mathbb{E}( G^2 | \mathcal{F}_s)$. Applying It\^o's formula to $N_s \ln N_s$ and taking expectation yields
 \[
 \mathbb{E}_x( N_t \ln N_t)-\mathbb{E}_x( N_0 \ln N_0)=\frac{1}{2} \mathbb{E}_x\left( \int_0^t \frac{d[N]_s}{N_s} \right),
 \]
 where $[N]$ is the quadratic variation of $N$. From Proposition \ref{clark-ocone}, we have
 \[
dN_s=2 \left\langle \mathbb{E} \left( G \tilde{D}^{\varepsilon}_sG  \mid \mathcal{F}_s \right), \para _{0,s} dB_s \right \rangle.
 \]
 Thus we have from Cauchy-Schwarz inequality
  \begin{align*}
 \mathbb{E}_x( N_t \ln N_t)-\mathbb{E}_x( N_0 \ln N_0) & \le 2 \mathbb{E}_x\left( \int_0^t \frac{\|\mathbb{E} \left( G \tilde{D}^{\varepsilon}_sG \mid \mathcal{F}_s \right)\|_{\varepsilon}^2}{N_s} ds \right) \\
  & \le 2 \mathbb{E}_x \left(\int_0^T \|\tilde{D}^{\varepsilon}_sG\|_{\varepsilon}^2ds\right).
 \end{align*}
\end{proof}

The second step is the following lemma.

\begin{lemma}\label{damped gradient bounded by summation}
With  $G=f(X_{t_1},X_{t_2},\cdots,X_{t_n})$ we have the following inequality
\begin{equation*}
\mathbb{E}_x \left(\int_0^T\|\tilde{D}^{\varepsilon}_sG\|_{\varepsilon}^2ds\right) \leq e^{T(K+\frac{\kappa}{\varepsilon})}\sum_{l=1}^n \frac{t_l-t_{l-1}}{T} \|\sum_{i=l}^n(\tau_{t_l}^{\varepsilon})^{-1}\tau_{t_i}^{\varepsilon}d_if(X_{t_1},\cdots,X_{t_n})\|_{\varepsilon}^2
\end{equation*}
\end{lemma}
\begin{proof}
From definition  \ref{gradient definition}, we have
\begin{equation*}
\tilde{D}_s^{\varepsilon}G=\sum_{i=1}^{n}\sum_{l=1}^i\mathbf{1}_{[t_{l-1},t_l]}(s)(\tau_s^{\varepsilon})^{-1}\tau_{t_i}^{\varepsilon}d_if(X_{t_1},\cdots,X_{t_n})=\sum_{l=1}^n\sum_{i=l}^n\mathbf{1}_{[t_{l-1},t_l]}(s)(\tau_s^{\varepsilon})^{-1}\tau_{t_i}^{\varepsilon}d_if(X_{t_1},X_{t_2},\cdots,X_{t_n}).
\end{equation*}
Thus we have 
\begin{equation*}
\mathbb{E}_x \left( \int_0^T \|\tilde{D}^{\varepsilon}_sG\|_{\varepsilon}^2ds\right)=\sum_{l=1}^n\int_{t_{l-1}}^{t_l}\|\sum_{i=l}^n(\tau_s^{\varepsilon})^{-1}\tau_{t_l}^{\varepsilon}(\tau_{t_l}^{\varepsilon})^{-1}\tau_{t_i}^{\varepsilon}d_if(X_{t_1},X_{t_2},\cdots,X_{t_n})\|_{\varepsilon}^2ds
\end{equation*}
By our definition of $\tau_t^{\varepsilon}$ and the assumption \ref{geometric assumption}, we have
\begin{equation*}
\|(\tau_s^{\varepsilon})^{-1}\tau_{t_l}^{\varepsilon}\sum_{i=l}^n(\tau_{t_l}^{\varepsilon})^{-1}\tau_{t_i}^{\varepsilon}d_if(X_{t_1},X_{t_2},\cdots,X_{t_n})\|_{\varepsilon}^2\leq e^{(K+\frac{\kappa}{\varepsilon})(t_l-s)}\|\sum_{i=l}^n(\tau_{t_l}^{\varepsilon})^{-1}\tau_{t_i}^{\varepsilon}d_if(X_{t_1},X_{t_2},\cdots,X_{t_n})\|_{\varepsilon}^2.
\end{equation*}
Hence
\begin{equation*}
 \mathbb{E}_x \left( \int_0^T \|\tilde{D}^{\varepsilon}_sG\|_{\varepsilon}^2ds\right) \leq \sum_{l=1}^n \int_{t_{l-1}}^{t_l}e^{(K+\frac{\kappa}{\varepsilon})(t_l-s)}ds\|\sum_{i=l}^n(\tau_{t_l}^{\varepsilon})^{-1}\tau_{t_i}^{\varepsilon}d_if(X_{t_1},X_{t_2},\cdots,X_{t_n})\|_{\varepsilon}^2.
\end{equation*}
We now use the elementary  inequality $\frac{e^{sc}-1}{c}\leq se^c$ to get :
\begin{equation*}
 \mathbb{E}_x \left( \int_0^T \|\tilde{D}^{\varepsilon}_sG\|_{\varepsilon}^2ds\right) \leq e^{T(K+\frac{\kappa}{\varepsilon})}\sum_{l=1}^n \frac{t_l-t_{l-1}}{T} \|\sum_{i=l}^n(\tau_{t_l}^{\varepsilon})^{-1}\tau_{t_i}^{\varepsilon}d_if(X_{t_1},X_{t_2},\cdots,X_{t_n})\|_{\varepsilon}^2.
\end{equation*}
\end{proof}

The third and final step is the following bound.

\begin{lemma}
With  $G=f(X_{t_1},X_{t_2},\cdots,X_{t_n})$, we have
\begin{equation*}
\sum_{l=1}^n \frac{t_l-t_{l-1}}{T}\|\sum_{i=l}^n(\tau_{t_l}^{\varepsilon})^{-1}\tau_{t_i}^{\varepsilon}d_if(X_{t_1},\cdots,X_{t_n})\|_{\varepsilon}^2\leq e^{2T(K+\frac{\kappa}{\varepsilon})}
 \mathbb{E}_x \left( \int_0^T\|D^\varepsilon_sG\|_{\varepsilon}^2ds\right)
\end{equation*}
\end{lemma}
\begin{proof}
From Definition \ref{gradient definition}, we have
\[
D^\varepsilon_tF=\sum_{i=1}^n\mathbf{1}_{0,t_i}(t)\Theta_{t_i}^{\varepsilon}d_if(X_{t_1},\cdots,X_{t_n})
\]
Also recall from equation (\ref{tau=M Theta}) that $\tau^{\varepsilon}_t=\mathcal{M}_{t}^{\varepsilon}\Theta_{t}^{\varepsilon}$.
Let us denote $z_l=\sum_{i=l}^n\Theta_{t_i}^{\varepsilon}d_if(X_{t_1},\cdots,X_{t_n})$. 
we then have 
\begin{align*}
 & \|\sum_{i=l}^n(\tau_{t_l}^{\varepsilon})^{-1}\tau_{t_i}^{\varepsilon}d_if(X_{t_1},\cdots,X_{t_n})\|_{\varepsilon}^2 \\
=& \|(\Theta_{t_l}^{\varepsilon})^{-1}\sum_{i=l}^n(\mathcal{M}_{t_l}^{\varepsilon})^{-1}\mathcal{M}_{t_i}^{\varepsilon}\Theta_{t_i}^{\varepsilon} f(X_{t_1},\cdots,X_{t_n})\|_{\varepsilon}^2\\
=& \left\|(\Theta_{t_l}^{\varepsilon})^{-1}\left( z_l+\sum_{i=l+1}^n[(\mathcal{M}_{t_l}^{\varepsilon})^{-1}\mathcal{M}_{t_i}^{\varepsilon}-(\mathcal{M}_{t_l}^{\varepsilon})^{-1}\mathcal{M}_{t_{i-1}}^{\varepsilon}]z_i\right) \right\|_{\varepsilon}^2 \\
=&\left\| z_l+\sum_{i=l+1}^n[(\mathcal{M}_{t_l}^{\varepsilon})^{-1}\mathcal{M}_{t_i}^{\varepsilon}-(\mathcal{M}_{t_l}^{\varepsilon})^{-1}\mathcal{M}_{t_{i-1}}^{\varepsilon}]z_i\right\|_{\varepsilon}^2
\end{align*}
From equation (\ref{multiplicative function M_t}) and  (\ref{geometric assumption}) we have
\begin{equation*}
\|[(\mathcal{M}_{t_l}^{\varepsilon})^{-1}\mathcal{M}_{t_i}^{\varepsilon}-(\mathcal{M}_{t_l}^{\varepsilon})^{-1}\mathcal{M}_{t_{i-1}}^{\varepsilon}]z_i\|_{\varepsilon}^2\leq \left(\frac{K+\frac{\kappa}{\varepsilon}}{2}\int_{t_{i-1}}^{t_i}e^{\frac{1}{2}(K+\frac{\kappa}{\varepsilon})(s-t_l)}ds\right)^2\|z_i\|_{\varepsilon}^2
\end{equation*}
Thus by Cauchy-Schwarz inequality, with $c=(K+\frac{\kappa}{\varepsilon})$ and $\lambda=\frac{c}{2}e^{\frac{c}{2}}$
\begin{equation*}
\|\sum_{i=l}^n(\tau_{t_l}^{\varepsilon})^{-1}\tau_{t_i}^{\varepsilon}d_if(X_{t_1},\cdots,X_{t_n})\|_{\varepsilon}^2\leq (1+\lambda)\|z_l\|_{\varepsilon}^2+(1+\frac{1}{\lambda})\frac{c^2}{4} \|\int_{t_{l}}^{T}e^{\frac{1}{2}c(s-t_l)}g_sds\|_{\varepsilon}^2
\end{equation*}
where $g_s=\|z_l\|_{\varepsilon}$ for $s\in [t_{l-1},t_l)$.  We easily deduce from that (see the argument in  Lemma 4.3 in \cite{Hsu} for more details), 
\begin{equation*}
\sum_{l=1}^n\frac{t_l-t_{l-1}}{T}\|\sum_{i=l}^n(\tau_{t_l}^{\varepsilon})^{-1}\tau_{t_i}^{\varepsilon}d_if(X_{t_1},\cdots,X_{t_n})\|_{\varepsilon}^2\leq  e^{2T(K+\frac{\kappa}{\varepsilon})}\int_0^Tg_s^2ds.
\end{equation*}
We then complete the proof by observing that
\begin{equation*}
\int_0^Tg_s^2ds=\mathbb{E}_x \left(\int_0^T\|D^\varepsilon_sG\|_{\varepsilon}^2ds\right).
\end{equation*}
\end{proof}

\section{Concentration inequalities}

In this section, we study concentration inequalities for the horizontal Brownian motion $(X_t)_{t \ge 0}$ which is started at a fixed point $x \in \M$. As it is well-known, see \cite{Led}, concentration inequalities are closely related to log-Sobolev inequalities.

We assume throughout the section that for every horizontal one-form $\eta\in \Gamma^{\infty}(\mathcal{H}^*)$,
\begin{equation}\label{bounds2}
\langle \mathfrak{Ric}_{\mathcal{H}}\eta,\eta\rangle_{\mathcal{H}}\geq -K\|\eta\|^2_{\mathcal{H}}, ~~~-\langle\mathbf{J}^2(\eta),\eta\rangle_{\mathcal{H}}\leq \kappa \|\eta\|_{\mathcal{H}}^2,
\end{equation}
with $K \ge 0,\kappa >0$.

For $\varepsilon >0$, we denote by $d_\varepsilon$ the distance associated with the Riemannian metric $g_\varepsilon$. To get a concentration bound for the distance $d_\varepsilon$, we adapt an argument from Ledoux \cite{Led} (see also  \cite{Hou}). 

\

The following lemma is a consequence of Herbst argument (see \cite{Led} page 148) applied to the log-Sobolev inequality in Lemma \ref{log sob damped}.

\begin{lemma}
Let $\varepsilon >0$. Let $F \in \mathbf{Dom} ( \tilde{D}^\varepsilon )$. If there is a constant $C>0$ such that $$\int_0^T \| \tilde{D}_s^\varepsilon F \|^2 ds < C,$$ almost surely, then for every $r \ge 0$,
\[
\mathbb{P}_x \left( F -\mathbb{E}_x (F) \ge r \right) \le \exp \left( - \frac{r^2}{2 \sigma^2} \right),
\]
where $\sigma^2=\mathbb{E}_x \left(\int_0^T \| \tilde{D}_s^\varepsilon F \|^2 ds \right)$.
\end{lemma}

The previous lemma implies the following concentration inequality.

\begin{proposition}
Let $\varepsilon >0$. We have for every $T>0$ and $r  \ge 0$
\begin{equation}
\mathbb{P}_x \left(\sup_{0\leq t\leq T}d_{\varepsilon}(X_t,x)\geq \mathbb{E}_x\left[\sup_{0\leq t\leq T}d_{\varepsilon}(X_t,x)\right]+r\right)\leq \exp\left(-\frac{r^2}{2Te^{\left(K+\frac{\kappa}{\varepsilon}\right)T}}\right).
\end{equation}
\end{proposition}
\begin{proof}
Let 
\[
F=f(X_{t_1},\cdots,X_{t_n})=\max_{1\leq i\leq n}d_{\varepsilon}(X_{t_i},x)
\]
where $0 \le t_1 \le \cdots \le t_n$ is a partition of $[0,T]$.

By using the arguments of \cite{Led}, p.196, we obtain for a certain partition $(A_j)_{1\leq j\leq n}$ of the path space,
\begin{align*}
\|\tilde{D}^\varepsilon_sF\|_{\varepsilon} & \leq \sum_{l=1}^{n}\mathbf{1}_{(t_{l-1},t_l]}(s)\sum_{i=l}^ne^{\frac{1}{2}(K+\frac{\kappa}{\varepsilon})(t_i-s)}\|d_if(X_{t_1},\cdots,X_{t_n})\|_{\varepsilon} \\
 & \leq e^{\frac{1}{2}(K+\frac{\kappa}{\varepsilon})T}\sum_{l=1}^n\mathbf{1}_{t_{i-1},t_i]}(s)\sum_{i=l}^n\mathbf{1}_{A_i}\leq e^{\frac{1}{2}(K+\frac{\kappa}{\varepsilon})T}
\end{align*}
 We then use the previous lemma and finish the proof by monotone convergence when the mesh of the partition goes to zero.
\end{proof}

The previous proposition easily implies that
\[
\lim \sup_{r \to \infty} \frac{1}{r^2} \ln \mathbb{P}_x \left(\sup_{0\leq t\leq T}d_{\varepsilon}(X_t,x)\geq r  \right) \le -\frac{1}{2Te^{\left(K+\frac{\kappa}{\varepsilon}\right)T}}
\]
Under  further assumptions we can also provide a lower bound. 

\begin{proposition}

Assume that \eqref{bounds2} is satisfied with $K=0$ and moreover that for any vector field $Z$,
\[
-\frac{1}{4} \mathbf{Tr}_\mathcal{H} (J^2_{Z})\ge \rho_2 \| Z \|^2_\mathcal{V},
\]
where $\rho_2>0$. Then for every $\varepsilon , T >0$,
\[
\lim \inf_{r \to \infty} \frac{1}{r^2} \ln \mathbb{P}_x \left(\sup_{0\leq t\leq T}d_{\varepsilon}(X_t,x)\geq r  \right) \ge -\frac{1}{T}\left(\frac{D}{n} + \frac{4 \varepsilon^2}{T} \frac{3D}{2\rho_2n} \ln \left(2 \right) \right)
\]
where $D= \left(
1+\frac{3\kappa}{2\rho_2}\right)n$.
\end{proposition}

\begin{proof}
Let us denote by $p(x,y,t)$ the heat kernel of the Markov process $(X_t)_{t\ge 0}$. It was proved in \cite{BBGM} that, under the above assumptions, it satisfies the following (non-optimal) lower bound
\begin{align*}
p(x,y,t) \geq  \frac{ C}{ \mu\left(B\left(x,\sqrt t \right)\right)}
\exp\left(-\frac{d_\varepsilon \left(y,x\right)^{2}}{t}\left(\frac{D}{n} + \frac{4 \varepsilon^2}{t} \frac{3D}{2\rho_2n} \ln \left(2 \right) \right)\right).
\end{align*}
The result then easily follows from a similar  argument as in \cite{Led} page 197.
\end{proof}

Since the horizontal distribution is bracket generating, then the family of Riemannian metrics converges when $\varepsilon \to 0$ to the sub-Riemannian distance that we denote by $d$. This distance is the distance intrinsically associated with the Dirichlet form associated to $X$. For this distance, we can get the following tail estimates.

\begin{proposition}
Assume that \eqref{bounds2} is satisfied with $K=0$ and moreover that for any vector field $Z$,
\[
-\frac{1}{4} \mathbf{Tr}_\mathcal{H} (J^2_{Z})\ge \rho_2 \| Z \|^2_\mathcal{V},
\]
where $\rho_2>0$. Then for every $T >0$,
\[
\lim \sup_{r \to +\infty} \frac{1}{r^2} \ln \mathbb{P}_x \left( \sup_{0 \le t \le T} d(X_t,x) \ge r \right) \le -\frac{1}{2T}
\]
and 
\[
\lim \inf_{r \to +\infty} \frac{1}{r^2} \ln \mathbb{P}_x \left( \sup_{0 \le t \le T} d(X_t,x) \ge r \right) \ge -\frac{D}{2nT},
\]
where  $D= \left(
1+\frac{3\kappa}{2\rho_2}\right)n$.
\end{proposition}

\begin{proof}
For the upper bound, we first observe that our assumptions imply that the metric measure space $(\M,d,\mu)$ satisfies the volume doubling property and that the uniform scale invariant Poincar\'e inequality on balls is satisfied (see \cite{BBG}). The upper bound is therefore a consequence of Theorem 3.9 in \cite{BSC}.

For the lower bound, it was also proved in \cite{BBG} that for any $0<\ve <1$
there exists a constant $C(\ve) = C(d,\kappa,\rho_2,\ve)>0$, which tends
to $\infty$ as $\ve \to 0^+$, such that for every $x,y\in \bM$
and $t>0$ one has
\[
\frac{C(\ve)^{-1}}{\mu(B(x,\sqrt
t))} \exp
\left(-\frac{D d(x,y)^2}{n(2-\ve)t}\right)\le p(x,y,t)\le \frac{C(\ve)}{\mu(B(x,\sqrt
t))} \exp
\left(-\frac{d(x,y)^2}{(2+\ve)t}\right).
\]
This lower bound on the heat kernel implies the expected result.
\end{proof}

\end{document}